\documentclass[12pt,twoside]{article}
\usepackage{amsmath,amsfonts}

\newtheorem{thm}{Theorem}[section]

\date{}

\begin{document}

\centerline{}
\centerline{\Large{\bf On algebraic independence of three p-adic continued
 fractions}}

\centerline{}


\newcommand{\mvec}[1]{\mbox{\bfseries\itshape #1}}

\centerline{\bf {Sarra Ahallal, Mohamed Begare and Ali Kacha}}

\centerline{}

\centerline{Department of Mathematics, Faculty of Science,}

\centerline{Ibn Tofail University, 14 000 Kenitra,  Morocco  }

\centerline{e-mails: sarraahallal@uit.ac.ma, mohamed.begare@uit.ac.ma, ali.kacha@uit.ac.ma }

\centerline{}

\newtheorem{Theorem}{\quad Theorem}[section]

\newtheorem{Definition}[Theorem]{\quad Definition}

\newtheorem{Corollary}[Theorem]{\quad Corollary}

\newtheorem{Lemma}[Theorem]{\quad Lemma}

\newtheorem{Example}[Theorem]{\quad Example}

{\bigskip}

\begin{abstract}
\noindent \textbf{\emph{In this paper, we establish sufficient conditions on the elements of the p-adic continued fractions
$A$ and $B$  which guarantee that the p-adic continued fractions
$A, B $ and $A^{B}$ are algebraically independent over $\mathbb{Q}$. These elements have partial quotients that increase rapidly. We
note that these results extend some work of Bundschuh. Furthermore, we give some numerical examples which illustrated the theoretical results.
}}
\end{abstract}

\noindent {\bf Key words:} \emph{p-adic continued fraction, rational approximation, algebraic independence.}\\
{\textbf 2020 Mathematics Subject Classification:} 11J61, 11J70,11J81
11J85.

\section{Introduction }

{\bigskip }

\noindent
The theory of algebraic independence of numbers has a long history.
In the case of p-adic numbers, a parallel approach to algebraic independence has been developed based on Diophantine approximation techniques adapted to the metric
p-adic metric. \\
In particular, the work of K. Mahler introduced a theory p-adic theory of transcendence that plays a similar role to that of Liouville and Roth in the real case. The transcendence of certain fast-growing continued fractions has been studied in the real setting by several authors, notably P.Bundschuh $[5]$, A. Durand $[6]$, W. Lianxiang $[7]$, G. Nettler
$[10]$ and T. Okano $[11]$.
\\
 These questions naturally arise in the context p-adic framework, where the structure of continued fractional developments and power series has specific properties.\\
Let $A$ and $B$ be
two continued fractions which are defined by
$$
A= a_{0}+\frac{1|}{|a_{1}}+ \frac{1|}{|a_{2}}+\cdot \cdot \cdot
$$
and
 $$
 B=b_{0}+\frac{1|}{|b_{1}}+ \frac{1|}{|b_{2}}+\cdot \cdot
\cdot
$$

\noindent where $a_{i}>0$, $b_{i}>0 $ are integers for any $i\geq 1.$
\\

In $[4],$ we have proved the transcendence and the algebraic independence of $A$ and $B$ when their partial quotients $a_n$ and $b_n$ are unbounded
by using Roth's approximation theorem in the real setting.
\\
We recall that in $[5]$, P. Bundschuh pointed out that no proof of the transcendence of $A^B$ could be obtained by applying Roth's approximation theorem alone. He then proposed a theorem to establish the transcendence of $A^B$ by another method. \\
So, in $[1],$ we have proved the transcendence of the real $A^B$  and in $[3]$  we get the algebraic independence of an arbitrary number of real continued fractions $A_1, A_2,..,A_k$ with restrictive assymptions
\\
We are interested here in the transposition of this approach to the p-adic framework, where Diophantine approximation techniques require adaptation to the ultrametric properties of p-adic numbers.
So, in this case, these questions are even trickier, since the nature of the approximations p-adic approximations and the constraints imposed by non-archimedianity profoundly modify transcendence and algebraic independence analysis.\\
We note that in $[2],$ we have proved the transcendence of the p-adic continued $A^B$ with unbounded p-adic absolute value of partial quotients. \\
\\
The main goal of this paper, is to prove the algebraic independence of $A, B$ and $A^B$ three p-adic continued fractions with unbounded p-adic absolute value of partial quotients. \\
We also remark that in their paper $[9]$ Longhi-Murru-Saettone and in its paper $[12]$ Ooto studied analogy of Baker's transcendence criterion for Ruban continued fractions with bounded p-adic absolute value of partial quotients or which are a sequence begining with arbitrarily by long palindromes.
\\
\\
\section{  Preliminaries and auxiliary results  }

\subsection{D\'efinition of a p-adic continued fraction}

Let $p$ be an odd prim. We denote by $|.|_p$ and $|.|$ the p$-$ adic absolute value and the Euclidean
 absolute value, respectively. We next present the algorithm of a $p-$ adic continued fraction expansion defined by Lianxing $[7]$ which gives a simpler way than of Shneider $[13]$.
 \\
Let  $\xi \in \mathbb{Q}_p,$ using the Hensel expansion of $\xi,$ we find
$$
\xi=\xi_0= c_{0,-m_0}p^{-m_0}+c_{0,-m_0+1}p^{-m_0+1}+...+c_{0,-1}p^{-1}+c_{0,0}+c_{0,1}p+c_{0,2}p^2+...+c_{0,n}p^n+...\
$$
where $m_0 \in \mathbb{N}, 0 \leq c_{0,k} \leq p-1.$\\
\noindent Let
$$a_0=c_{0,-m_0}p^{-m_0}+c_{0,-m_0+1}p^{-m_0+1}+...+c_{0,-1}p^{-1}+c_{0,0},
$$

\noindent since $0 \leq c_{0,k} \leq p-1, $ for all  $-m_0 \leq k \leq 0, $ we have $0 \leq a_0 < p.$
\\
\noindent Then, write $\xi$ in the form
$$
\xi=a_0+\displaystyle{\frac{1}{\xi_1}}.
$$
\noindent So, we have $|\xi_1|_p \geq p$ and $\xi_1$ has the form
$$
\xi_1=c_{1,-m_1}p^{-m_1}+...+c_{1,-1}p^{-1}+c_{1,0}+c_{1,1}p+...+ c_{1,n}p^n+...,
$$
\noindent where $m_1 \geq 1$ and $c_{1,-m_1} \neq 0.$ We also define
$$
a_1=c_{1,-m_1}p^{-m_1}+...+c_{1,-1}p^{-1}+c_{1,0}, \text{ so \ } |a_1|_p=|\xi_1|_p \geq p.
$$

\noindent We continue this process and in general, for $k \geq 1,$  we put
$$
\xi_k=a_k +\frac{1}{\xi_{k+1}}=c_{k,-m_k}p^{-m_k}+...+c_{k,-1}p^{-1}+c_{k,0}+c_{k,1}p+c_{k,2}p^2+...+ c_{k,n}p^n+...,
$$
$$\text{ and \ } a_k=c_{k,-m_k}p^{-m_k}+...+c_{k,-1}p^{-1}+c_{k,0}, \ \text{ with \ } |a_k|_p=|\xi_k|_p=p^{m_k} \geq p \text{ and \ }  0 < a_k <p.$$

\noindent As in the real case, $\xi$ can be written briefly in the form

$$
\xi=[a_0; a_1,a_2,...,a_n,\xi_{n+1}], \ \text{ where \ }
$$
\begin{equation*}
a_n \in  \mathbb{Z}[1/p] \cap (0,p) \ \text{ for \ }  n \geq 1 \text{ and } a_0 \in  \mathbb{Z}[1/p] \cap [0,p).
\end{equation*}

{\bigskip}

\subsection{ Convergents of a p-adic continued fraction}

{\bigskip }

\noindent With the above notations, we define two sequences $(p_n)_{n \geq -1}$ and $(q_n)_{n \geq -1}$ by
\\
\begin{equation*}
\left\{
\begin{array}{c}
p_{-1}=1,\;p_0=a_0 \\
q_{-1}=0,\;q_0=1
\end{array}
\right. \mbox{ and } \ \left\{
\begin{array}{c}
p_n=a_n \;p_{n-1}+ p_{n-2} \\
q_n=a_n \;q_{n-1}+q_{n-2}
\end{array}
\right.  \; n\geq 1.
\end{equation*}

\noindent The elements $p_n$ and $q_n$ are polynomials in $a_0,a_1,a_2,..,a_n $ with coefficients equal to $1.$

{\bigskip }

\noindent {\bf Definition. \ }
\noindent Consider the continued fraction $\xi = [a_0; a_1,a_2,..].$ where $a_k$ are p-adic integer numbers.\\
Let  $\dfrac{P_n}{Q_n}=[a_0; a_1, a_2,...,a_n]$ with
$P_n \in \mathbb{Z},  $ $ Q_n \in \mathbb{N^*}$ and $\gcd(P_n,Q_n)=1.$
\\
\noindent The fraction $\dfrac{P_n}{Q_n}$ is called the $n^{th }$ convergent of the p-adic continued fraction $\xi.$
{\bigskip }

\noindent The main properties of convergents are as follows.

{\bigskip }

\noindent {\bf Lemma 1. \ }
 With the same notations as above, one has
\\
(i)
\begin{equation*}
 |p_n|_p=
\left\{
\begin{array}{c}
|a_0a_1...a_n|_p, \ n \geq 0  \ \hbox{  \ if \ } a_0 \neq 0,  \\
|a_2...a_n|_p  \hbox{ \ and \ } |p_1|_p=1 \  \hbox{  \ if \ } a_0 = 0.  \\
\end{array}
\right.
\end{equation*}

\noindent (ii)
$$
|q_n|_p= |a_1a_2...a_n|_p.
$$
\noindent (iii) $$
\xi -\frac{p_n}{q_n}=\frac{(-1)^n}{q_n(\xi_{n+1}q_n+q_n-1)}.
$$
\noindent (iv)
$$
|\xi-\frac{p_n}{q_n}|_p=\frac{1}{|q_n q_{n+1}|_p} \text{ \ so \ }
\lim_{n \rightarrow +\infty} \frac{p_n}{q_n}=\xi.
$$

\noindent{\bf Proof. } The equalities (i) and (ii) are easily checked by induction. For (iii) and (iv), see $[7].$ \\

{\bigskip}

\noindent {\bf Lemma 2. \ } {\it Let $\xi $ be un element of $\mathbb{Q}_p\setminus \mathbb{Q}$ such that $\xi=[a_0;a_1,...,a_n,...].$  then,
for all $n \geq 1,$
we have
\begin{equation*}
\left\{
\begin{array}{c}
0<p_n \leq (p+1)^{n+1},  \\
0 <p_n \leq (p+1)^{n}.
\end{array}
\right.
\end{equation*}
}
\noindent{\bf Proof. } By using definitions of $p_n$ and $q_n,$ we prove (i) and (ii).

{\bigskip}

\noindent {\bf Lemma 3. } (The algebraic form of Taylor development of order $n$ in $\mathbb{R}^2.$ )
\\
{\it For a development of order $n$, we use the sum of partial derivatives up to order $n$  near a given point $(a,b)$ :
$$f(x,y)=\sum_{k=0}^{n}\sum_{m=0}^{k} \dfrac{1}{m!(k-m)!} \dfrac{\partial^k f(a,b)}{\partial x^m \partial y^{k-m}}(x-a)^m(y-b)^{k-m}+R_n $$
where $R_n$ is the Lagrange remainder, which measures the approximation error.}

{\bigskip}

\section{Main result }

\subsection{ The p-adic continued fraction of $A, B$ and $A^B$}

\noindent Let $p$ be an odd prime integer, $$A\in (1+p\mathbb{Z}_p) \Leftrightarrow \mid A-1 \mid_p \leq p^{-1}<1,$$
$$B\in (p\mathbb{Z}_p) \Leftrightarrow \mid B \mid_p < p^{-1}<1.
$$
\\
\noindent Define the p-adic $\ln$ by $$\ln(A)=\sum_{n=1}^{+\infty} \dfrac{(-1)^{n-1}}{n}(A-1)^n.
$$
Since
 $\ln(A) \in(p\mathbb{Z}_p)$ and $B\in(p\mathbb{Z}_p),$  We then define $A^B$ as
 $$
 6A^B=e^{B \ln(A)}.
 $$
We note $A=[a_0,a_1,...,a_n,...]$ and $B=[b_0,b_1,...,b_n,...]$
the developments in p-adic continued fractions of $A$ and $B$ respectively. Their convergents are noted by
 $$
 A_n=\dfrac{^ap_n}{^aq_n}=[a_0,a_1,...,a_n],$$
$$B_n=\dfrac{^bp_n}{^bq_n}=[b_0,b_1,...,b_n].
$$
\noindent From Lemmas $1$ and $2$, it is easy to see that
\\
\begin{center}
$
\begin{cases}
\mid A_n \mid_p = \dfrac{\mid ^ap_n \mid_p}{\mid ^aq_n \mid_p}<\dfrac{1}{\mid a_1 \mid_p}<\dfrac{1}{p}\\
\mid B_n \mid_p = \dfrac{\mid ^bp_n \mid_p}{\mid ^bq_n \mid_p}<\dfrac{1}{\mid b_1 \mid_p}<\dfrac{1}{p}
\end{cases}
$
$\Longrightarrow$
$\begin{cases}
\mid ln(A_n) \mid_p <\dfrac{1}{p}\\
\mid B_n ln(A_n) \mid_p<\dfrac{1}{p^2}<\dfrac{1}{p}.
\end{cases}
$
\end{center}

\noindent We deduce that
$e^{B_n \ln (A_n)}$ is well defined, so we put $A_n^{B_n}=e^{B_n \ln(A_n)}.$

{\bigskip}

\noindent Our main result is given in the next Theorem.

{\bigskip}

\begin{thm}
 Let $A\in (1+p\mathbb{Z}_p), B\in (p\mathbb{Z}_p)$ and $(\alpha)$ be a real number $>2.$
\\
If $$ \mid a_n\mid_p \geq \mid b_n \mid_p >\mid a_{n-1}\mid_p^{\alpha}\;\; \text{for all } \;\;n\geq2,
$$
then the p-adic continued fractions $A,B$ and $A^B$ are algebraically independent over $\mathbb{Q}.$
\end{thm}

{\bigskip}

\noindent{\bf Remark. \ }
We note that, with the hypotheses of Theorem 3.1 on $a_n$ and $b_n,$ the p-adic continued fractions $A,B$ and $A^B$ are three transcendental numbers, see $[7]$ and $[2].$

{\bigskip }

\noindent{\bf Proof of Theorem 3.1 }
{\bigskip}

\noindent Assume that $A,B,A^B$ are not algebraic independent over $\mathbb{Q.}$ So, there exists a polynomial
 $ Q \in \mathbb{Z}[X,Y,Z]\setminus \{0\}$ with a of minimal degree $d$ such that
 $$ \; Q(A,B,A^B)=0.
 $$
\noindent Let $Q(X,Y,Z)=\sum_{k+l+m=0}^{d}e_(k,l,m)X^kY^l Z^m.$ Firstly, according to the  mean value theorem, we have\\
\begin{eqnarray*}
Q(A,B,A^B)- Q(A_n,B_n,{A_n}^{B_n})& = (A-A_n)\dfrac{\partial Q}{\partial x}(A_{\theta},B_{\theta},C_{\theta})
+ (B-B_n) \dfrac{\partial Q}{\partial y}(A_{\theta},B_{\theta},C_{\theta}) \\
\\
                                   &+ (A^B-A_n^{B_n}) \dfrac{\partial Q}{\partial z}(A_{theta},B_{theta},C_{\theta})
\end{eqnarray*}
where $(A_{\theta},B_{\theta},C_{\theta})=(A+\theta(A-A_n),B+\theta(B-b_n),A^B+\theta(A^B-A_n^{B_n})).$ \\
\\
So, we deduce that there exists a positive constant $C=C(Q,A,B)$ such that
\begin{equation}
|Q(A,B,A^B)- Q(A_n,B_n,{A_n}^{B_n}) |_p \leq C \max( |A-A_n|_p, |B-B_n|_p, |A^B-A_n^{B_n}|_p)
\end{equation}

Secondly, let $f(x,y)=x^y$ with $x$ and $y$ are strictly positive and applying the real Taylor theorem to $f$ in $(A_n,B_n)$, we get:
\begin{align*}
f(x,y) =\sum_{\nu_1+\nu_2\geq1}C_{(\nu_1,\nu_2)}(x-A_n)^{\nu_1}(y-B_n)^{\nu_2}
\end{align*}

\noindent Let $C_i=C_{(\nu_1,\nu_2)}$ for $i\in\{1,2\}$,
with $(\nu_1,\nu_2)$ admits $1$ in the $i^{th}$ coordinate and $0$ in the other case.\\
\\
We have $C_1=\dfrac{\partial f}{\partial x}(A_n,B_n)=\frac{B_n}{A_n} A_n^{B_n} $ and
$C_2=\dfrac{\partial f}{\partial y}(A_n,B_n)=ln(A_n)A_n^{B_n}\neq 0.$
\\
So we get,
 $$A^B-A_n^{B_n}=C_1 (A-A_n)+ C_2 (B-B_n) +O(\mid B-B_n\mid)
 $$
In addition, introducing its p-adic absolute value, there exists a positive and bounded real constant
 $C_3=C(A,B)= max(\mid C_1 \mid_p, \mid C_2 \mid_p)<1/p$
such that
\begin{align*}
\mid A^B-A_n^{B_n}\mid_p &\leq C_3.max(\mid A-A_n\mid_p,\mid\ B-B_n \mid_p) \\
                          &  \leq C_3.\mid B-B_n \mid_p = \dfrac{C_3}{\mid \;^bq_n\;.\;^bq_{n+1}\mid_p}.
\end{align*}

It yields that
\begin{equation}
\mid A^B-A_n^{B_n}\mid_p\leq \dfrac{C_3}{\mid \;^bq_n\;.\;^bq_{n+1}\mid_p}.
\end{equation}
\\
In the other hand, we have
$$Q(A_n,B_n,A_n^{B_n})=\sum_{k+l+m=0}^d e_(k,l,m)(\dfrac{\;^ap_n}{\;^aq_n})^{k+mB_n}(\dfrac{\;^bp_n}{\;^bq_n})^{l}.
$$
It follows that
$$
|Q(A_n,B_n,A_n^{B_n})| \leq \sum_{k+l+m=0}^d |e_(k,l,m)|max (\dfrac{\;^ap_n}{\;^aq_n},\dfrac{\;^bp_n}{\;^bq_n})^{k+l+mB_n}.
$$

\noindent We note that $^ap_n, ^aq_n, ^bp_n$ and $^bq_n$ are rational numbers, so let $P_n\in \mathbb{N}$ and $Q_n\in\mathbb{N}^*$ with $gcd(P_n,Q_n)=1$
such that $ max( \dfrac{\;^ap_n}{\;^aq_n}, \dfrac{\;^bp_n}{\;^bq_n}) =\dfrac{P_n}{Q_n}.$
\\
Then, we find $$Q(A_n,B_n,A_n^{B_n})=\sum_{k+l+m=0}^de_(k,l,m)\dfrac{P_n^{k+l+mB_n}}{Q_n^{k+l+mB_n}}.
$$
Therefore, there exists a real and bounded constant \\
$ C_4=C(A,B)=(d+1)H(Q)$ where $ H(Q)= max |e_(k,l,m)| $ such that
\begin{equation}
    \mid Q_n^{dB_n}.Q(A_n,B_n,A_n^{B_n})\mid \leq C_4 (max(P_n,Q_n))^{dB_n} \leq C_4P_n^{dB_n} <C_4P_n^d,
\end{equation}

\noindent By using the product formula
$$ \mid Q_n^{dB_n}.Q(A_n,B_n,A_n^{B_n})\mid_p.\mid Q_n^{dB_n}.Q(A_n,B_n,A_n^{B_n}) \mid\geq 1$$
and relationship $(3),$  we obtain
 \begin{equation*}
  \mid Q_n^{dB_n}.Q(A_n,B_n,A_n^{B_n})\mid_p\geq \dfrac{1}{\mid Q_n^{dB_n}.Q(A_n,B_n,A_n^{B_n})\mid}>\dfrac{1}{C'P_n^d}.
\end{equation*}
Then we would have
 \begin{equation}
\dfrac{1}{C_4P_n^d\mid Q_n^{dB_n}\mid_p } <\mid Q(A_n,B_n,A_n^{B_n})\mid_p.
\end{equation}

\noindent According to results $(1)$, $(2)$ and $(4)$ we get
 \\
$$\dfrac{1}{C_4P_n^d\mid Q_n^{dB_n}\mid_p } <\mid Q(A_n,B_n,A_n^{B_n})\mid_p< \dfrac{C_3}{\mid \;^bq_n\;.\;^bq_{n+1}\mid_p}
$$
As, $Q_n\in\mathbb{N}^*$, therefore $\mid Q_n\mid_p=
\begin{cases}
1\;\;if\;\; p\; \text{does not divide }\;Q_n\\
p^{-n_0}\;\;\text{with}\;\;n_0\in\mathbb{N}^*\;\;\text{if}\;\;p/Q_n
\end{cases}$
\\
which gives,
$$\dfrac{1}{C_4P_n^dp^{n_0}}<\dfrac{C_3}{\mid\;^bq_n\;^bq_{n+1}\mid_p}.
$$
By (ii) of Lemma 2, we have $$\mid\;^bq_n\;^bq_{n+1}\mid_p=\mid b_1b_2...b_n\mid_p^2.\mid b_{n+1}\mid_p.
$$
It follows that there a real and bounded constant $C_5=CC_4>0$ such that,
$$
\mid b_1b_2...b_n\mid_p^2.\mid b_{n+1}\mid_p<C_5p^{n_0}P_n^d.
$$
Since $P_n \leq ^ap_n \leq (p+1)^n,$ it yields that
\begin{align}
\mid b_1b_2...b_n\mid_p^2.\mid b_{n+1}\mid_p \leq C_5p^{n_0}(p+1)^{nd}.
\end{align}
From the inequalities
$$\mid b_{n+1}\mid_p\geq \mid b_n\mid_p^{\alpha}\geq \mid b_{n-1}\mid^{\alpha^2}\geq...\geq \mid b_1\mid_p^{\alpha^n},$$
\\
the relationship $(5)$ becomes
$$ \mid b_1\mid_p^{2\sum_{k=0}^{n-1}\alpha^k+\alpha^n} \leq \mid b_1b_2...b_n\mid_p^2.\mid b_{n+1}\mid_p \leq C_6(p+1)^{nd}$$\
with $C_6=C_5p^{n_0}.$ Since $p\leq \mid b_1\mid_p,$ we finally get
$$
p^{2\sum_{k=0}^{n-1}\alpha^k+\alpha^n} \leq C_6 (p+1)^{nd}
$$
 which is not verified and leads to a contradiction,
because the inequality ${2\sum_{k=0}^{n-1}\alpha^k+\alpha^n} <nd$ is not correct for $\alpha >2 $ and $n$ sufficiently large.
\\
Therefore, we conclude that $A,B,A^B$ are algebraically independent over $\mathbb{Q}.$

{\bigskip}

\section{Numerical application}

{\bigskip}

Let $p=3,$ $\alpha = 3$ and

\begin{equation*}
\left\{
\begin{array}{c}
a_0=1,  \ a_{n}= 3^{-k_n-1},  \\
\\
b_0=0,  \ b_{n}=3^{-k_n}, \\
\\
\text{ where } \ k_1=1, \ k_n=3 k_{n-1}+4 \text{ for }  \  n \geq 1.  .
\end{array}
\right.
\end{equation*}
So, we have

$$\mid a_n \mid_3= 3^{k_n+1} >
\mid b_n \mid_3=3^{k_n}
$$
and
$$\mid b_n= \mid_3=3^{k_n} > \mid a_{n-1} \mid_3^{3} = 3^{k_n-1}.
$$
It follows from Theorem $2$ that $A=[1;3^{-8},3^{-26},3^{-80},...],$  $B=[0;3^{-7},3^{-25},3^{-79},...]$ and $A^B$
 are algebraically independent over $\mathbb{Q}.$

{\bigskip}

\noindent \textbf{Statements and Declarations}

{\bigskip}

\noindent \textbf{Declaration of interests tool. } I have nothing to declare.

{\bigskip}

\noindent \textbf{Funding sources. } This research did not receive any specific grant from funding agenciers in the public, commercial or not-for-profit-sectors.

{\bigskip}

\noindent \textbf{Data availability. } no data was used for the research described in the article.

{\bigskip}

\end{document}